\documentclass[leqno, 12pt]{article}%
\usepackage{amssymb}
\usepackage{amsmath}%
\usepackage{amsfonts}%
\usepackage{graphicx}
\usepackage{amsthm}%
\providecommand{\U}[1]{\protect\rule{.1in}{.1in}}
\usepackage[dvips]{color}

\newtheorem{thm}{Theorem}[section]
\newtheorem{prop}[thm]{Proposition \!\!}
\newtheorem{cor}[thm]{Corollary \!\!}
\newtheorem{lem}[thm]{Lemma \!\!}

\newtheorem{definition}[thm]{Definition}

\newcommand{\h}{\mathcal H}

\newcommand{\C}{\mathbb C}

\newcommand{\w}{\wedge}

\usepackage{graphicx}
\usepackage{rotating}
\usepackage{mathrsfs}
\usepackage{amsmath}
\usepackage{amsfonts}
\usepackage{amssymb}
\usepackage{amsthm}
\usepackage{tabularx}
\usepackage{color}
\usepackage{latexsym}
\usepackage{xypic}

\title{Spectral mapping theorem and the Taylor spectrum}
\author{Muneo Ch\=o, B. Na\v cevska Nastovska and K\^otar\^o Tanahashi }
\date{}

\begin{document}

\maketitle \footnotetext{2010 Mathematics Subject Classification;
Primary 47B20, Secondary 47A10.\par Keywords; Helbert space, Taylor spectrum, spectral mapping theorem, $p$-hyponormal, $\log$-hyponormal.\par  }

\begin{abstract} In \cite{CT} Ch\=o and Tanahashi showed new spectral mapping theorem of the Taylor spectrum for doubly commuting pairs 
of $p$-hyponormal operators and $\log$-hyponormal operators. In this paper,  we will show that same spectral mapping theorem holds for commuting $n$-tuples.
\end{abstract}

\

\

\section{Introduction and preparation }

Let $\h$ be a complex Hilbert space and $B(\h)$ be the set of all bounded linear operators on $\h$. For $T \in B(\h)$, let $\sigma(T), \sigma_p(T)$  and  $\sigma_a(T)$ denote the spectrum, the point spectrum and  the approximate point spectrum of  $T$, respectively. Let $\lambda \in \C$ belong to the residual spectrum  $\sigma_{r}(T)$ of $T$ if there exists $ c>0$ such that $ \Vert (T- \lambda) x \Vert \geq c \Vert x \Vert $ for all $x \in {\mathcal H}$ and $ (T-\lambda) {\mathcal H} \not= {\mathcal H}$.  It is easy to see that if  $\lambda \in \sigma_r(T)$, then $0 \in \sigma_p((T-\lambda)^*)$.  It is well known that $\sigma(T) = \sigma_a(T) \cup \sigma_r(T)$. For an Hermitian operator $A \in B(\h)$, we denote $A \geq 0$ if $(Ax, x) \geq 0$ for every $x \in \h$ and $A \geq B$ if $A - B \geq 0$. When $(Ax, x) > 0$ for every non-zero $x \in \h$, then we denote $T > 0$. For a given $p>0$, $T \in B(\h)$ is said to be $p$-hyponormal if  $(T^*T)^p \geq (TT^*)^p$. When $p=1/2$, $T$ is said to be semi-hyponormal. It means that  $T$ is semi-hyponormal if and only if $|T| \geq |T^*|$.  $T$ is said to be $\log$-hyponormal if  $T$ is invertible and $\log |T| \geq \log |T^*|$. It is well known that if $T$ is invertible $p$-hyponormal for some $ p>0$, then $T$ is $\log$-hyponormal. 
If ${\mathcal M}$ is a reducing subspace for a $p$-hyponormal or $\log$-hyponormal operator $T$, then so is  $T|_{{\mathcal M}}$, respectively.

\bigskip

For a commuting $n$-tuple ${\bf T} = (T_1,...,T_n) \in B(\h)^n$, we explain the Taylor spectrum  $\sigma({\bf T})$ of ${\bf T}$ shortly. Let $E^n$ be the exterior algebra on $n$ generators, that is, $E^n$ is the complex algebra with identity $e$ generated by indeterminates $e_1,..., e_n$. Let $E^n_k(\h) = \h \otimes E^n_k$.  Define $D^n_k \ : \ E^n_k(\h) \ \longrightarrow \ E^n_{k-1}(\h)$ by

$$ D^n_k (x \otimes e_{j_1} \w \cdots \w e_{j_k}) :=\sum_{i =1}^k (-1)^{i-1} T_{j_i} x \otimes e_{j_1} \w \cdots 
\w \check{e}_{j_i} \w \cdots \w e_{j_k}, $$ 
where $\check{e}_{j_i}$ means deletion. We denote $ D^n_k $ by $D_k$ simply. 
We think Koszul complex $E({\bf T})$ of ${\bf T}$ as follows: 

\

$ \ \ \ \ E({\bf T}) \ : \ 0 \longrightarrow E^n_n(\h) \stackrel{D_n}{\longrightarrow} E^n_{n-1}(\h) \ \stackrel{D_{n-1}}{ \longrightarrow} \ \cdots \stackrel{D_2}{\longrightarrow} E^n_1(\h) \stackrel{D_1}{\longrightarrow} E^n_0(\h) \longrightarrow 0.$ 
\vspace{2mm}

\noindent
Since $E^n_k(\h) \ \cong \  \overbrace{\h \oplus  \cdots \oplus \h}^{\binom{n}{k}=\frac{n !}{ (n-k) ! \, k !}}  \ \  (k=1,...,n),$ we set $E^n_k(\h) \ = \  \overbrace{\h \oplus  \cdots \oplus \h}^{\binom{n}{k}}  \ \  (k=1,...,n).$
\vspace{2mm}

\noindent
\begin{definition} A commuting $n$-tuple ${\bf T} = (T_1,...,T_n) \in B(\h)^n$ is said to be singular if and only if the Koszul complex $E({\bf T})$  of ${\bf T}$  is not exact.
\end{definition}

\noindent
\begin{definition} For a commuting $n$-tuple ${\bf T}= (T_1,...,T_n) \in B(\h)^n$, $z=(z_1,...,z_n) \in \C^n$ belongs to the Taylor spectrum $\sigma_T({\bf T})$ of $\bf{T}$ if ${\bf T} - z = (T_1 - z_1,...,T_n - z_n)$ is singular. 
\end{definition}

\noindent
About the definition of the Taylor spectrum, see details J. L. Taylor \cite{T1} and \cite{T2}. In \cite{Cu}, Curto proved the following proposition.

\begin{prop}\label{PC}{\rm (Proposition 3.4, Curto $\cite{Cu}$) } 
For a commuting $n$-tuple ${\bf T}= (T_1,...,T_n) \in B(\h)^n$, $0 = (0,...,0) \not \in \sigma_T({\bf T})$ if and only if $D_k^*D_k + D_{k+1} D_{k+1}^*$ is invertible for all $k$.

\end{prop}

\

For a commuting pair ${\bf T} = (T_1,...,T_n) \in  B(\h)^n$, it is well known that, for polynomials $f_1,...,f_m$ of $n$-variables, if $f(z_1,...,z_n)=(f_1(z_1,..., z_n),...,f_m(z_1,...,z_n))$, then it holds 
$$\sigma_T(f(T_1,...,T_n)) = f(\sigma_T(T_1,...,T_n)), $$
where $\sigma_T(T_1,...,T_n)$ is the Taylor spectrum of ${\bf T} = (T_1,...,T_n).$ See Theorem 4.7 in $\cite{T2}$.

\

In the paper \cite{CT}, Ch\=o and Tanahashi showed another spectral mapping theorem under the following assumption. \\

\noindent
Let $T=U|T| \in B(\h)$ be the polar decomposition of $T$ with unitary $U$ and $f$ be a continuous function on the non-negative real line which contains $\sigma(|T|)$. 
Let ${\mathcal K}$ be Berberian extension of $\h$ and $\circ : B(\h) \ni T \rightarrow T^{\circ} \in B({\mathcal K})$ be a faithful $*$-representation. We set the following conditions (1) and (2):

\begin{align}  & \text{For a sequence } \{x_n\} \text{ of unit vectors, if } \, (T- z) x_n \rightarrow 0, \ \text{then} \  (T - z)^* x_n \rightarrow 0. \\
                 & \text{If a closed subspace } \ {\mathcal M} \ \text{ of }  \ {\mathcal K} \ \text{reduces  } \ T^{\circ} \ \text{ and} \   re^{i\theta} \in
  \sigma(T^{\circ} \vert_{\mathcal M}),   \\
 & \qquad \text{ then } \  {\mathcal M} \ \text{ reduces  } \ U^{\circ}, \vert T \vert^{\circ} \ \text{and} \ 
 e^{-i \theta} f(r) \in \sigma_{p} \left(  ( U^{\circ} \vert_{\mathcal M} f( \vert T \vert^{\circ} ) 
\vert_{\mathcal M} )^{*} \right). \notag
 \end{align}

\begin{thm} \label{thm1} 

Let  ${\bf T} = (T_1,T_2)$ be a doubly commuting pair of operators and $T_j= U_j|T_j| \ (j=1,2)$ be the polar decomposition. Let $f(t)$ be a continuous function on a open interval in the non-negative real line which contains $\sigma(|T_1|) \cup \sigma(|T_2|)$. Let $S_j = U_j f( |T_j|) \ (j=1,2)$ and  ${\bf S} = (S_1,S_2)$. Let $T_1, T_2$ and $f$ satisfy $(1)$ and $(2)$. If  $  (r_1 e^{i \theta_1},r_2e^{i \theta_2}) \in \sigma_T({\bf T}), \ \mbox{then }  \  (e^{i \theta_1} f(r_1), e^{i \theta_2} f(r_2))  \in \sigma_T({\bf S}) . $

\end{thm}

See the details of Berberian extension \cite{B}. That proof depends on the following Vasilescu's result.

\   

\noindent
Let ${\bf T} = (T_1,T_2)$ be a commuting pair of operators on $\h$, ${\bf z} = (z_1,z_2) \in \C^2$ and let
$$\alpha({\bf T} -{\bf z}) := \left( \begin{array}{cc}
T_1-z_1 & T_2-z_2 \\
-(T_2-z_2)^* & (T_1 -z_1)^*
\end{array}
\right)  \ \ \mbox{on } \ \h \oplus \h.$$

\

\noindent
Then Vasilescu proved the following result.

\begin{prop} {\rm (Theorem 1.1, Vasilescu $\cite{V}$) } 
Let ${\bf T} = (T_1,T_2) \in B(\h)^2$ be a  commuting pair. Then    
$${\bf z}=(z_1, z_2) \in \sigma_T( {\bf T}) \ \mbox{if and only if }  \   \alpha({\bf T} -{\bf z})  \ \mbox{is not invertible}.$$  
\end{prop} 

\noindent
Therefore, we have  ${\bf z}=(z_1, z_2) \in \sigma_T( {\bf T}) \ \ \mbox{if and only if }  \   \ 0 \in \sigma(\alpha ({\bf T} -{\bf z})). $
\vspace{3mm}

\

\noindent
For an $n$-tuple ${\bf T} = (T_1,...,T_n)$, the joint point spectrum  $\sigma_{jp}({\bf T})$ is the set of all numbers ${\bf z} = (z_1,...,z_n) \in \C^n$ such that there exists a non-zero vector $x \in \h$ which satisfies $T_jx = z_jx \ (\forall j  =1,...,n)$ and the joint approximate point spectrum  $\sigma_{ja}({\bf T})$  is the set of all numbers ${\bf z} = (z_1,...,z_n) \in \C^n$ such that there exists a  sequence $\{x_k \}$ of unit vectors of $\h$ which satisfies
$$(T_j - z_j) x_k \ \to \ 0 \ \mbox{as}  \ k \to \infty \ (\forall j  =1,...,n).$$

\noindent
Following proposition is due to Berberian \cite{B} for a single operator case. It is easy to see a proof for $n$-tuples. See Berberian \cite{B} and Ch\=o \cite{C}.

\begin{prop} Let $B(\h)$ be the set of all bounded linear operators on $\h$. Then there exist an extension space ${\mathcal K}$ of $\h$ and a faithful $*$-representation of $B(\h)$ into $B({\mathcal K}) : \ T \ \to \ T^{\circ}$ such that

$$\sigma_{ja}({\bf T}) = \sigma_{ja}({\bf T}^{\circ}) = \sigma_{jp}({\bf T}^{\circ}), $$
where ${\bf T} = (T_1,...,T_n) \in B(\h)^n$ and ${\bf T}^{\circ} = (T_1^{\circ},...,T_n^{\circ})$.
 \end{prop}

\

\noindent
Following results are well known.

\begin{prop} \label{prop1} 

Let $T=U|T|$ be the polar decomposition of $T$ and $f$ be a continuous function on the non-negative real line which contains $\sigma(|T|)$.  For a sequence $\{ x_n \}$ of unit vectors, if $(T- r e^{i \theta}) x_n \rightarrow 0$ and $(T - r e^{i \theta})^* x_n \rightarrow 0$, then 
 $(U - e^{i\theta})x_n \rightarrow 0,  (|T| - r) x_n \rightarrow  0$ and $ (f(|T|) - f(r)) x_n \rightarrow  0$.

\end{prop}

\noindent
See Lemma 1.2.4 in $\cite{X1}$. 


\begin{prop}\label{prop2} 

Let $T$ be semi-hyponormal. Then $\sigma(T) = \{ \overline{z} : z \in \sigma_{a} ( T^{*} ) \}$.

\end{prop}

\noindent
See Theorem 1.2.6 in $\cite{X1}$.

\

\noindent
{\bf Remark.}
If $T$ is $p$-hyponormal  and $f(t) = t^{2p}$, then (2) holds by Theorem 4 of \cite{CH}.
 If $T$ is $\log$-hyponormal  and  $f(t) = \log \, t$, then (2) holds by  Lemma 3 of \cite{T}.
About (3), since the mapping $\circ$ of Berberian method is a faithful $*$-representation, so is $T^{\circ}$ if $T$ is $p$-hyponormal or $\log$-hyponormal, respectively. Let ${\mathcal M}$ be a reducing subspace for $T$. It is clear that if $T$ is $p$-hyponormal or $\log$-hyponormal, then so is  $T|_{{\mathcal M}}$, respectively. \\
(i) Let $T$ be $p$-hyponormal and $T= U|T|$ be the polar decomposition of $T$ and $f(t) = t^{2p}$.
 Then $S = U|T|^{2p}$ is semi-hyponormal and  $\sigma (U|T|^{2p}) = \{ r^{2p} e^{i\theta} \, : \, r e^{i \theta} \in \sigma(T) \, \}$ by Theorem 3 of \cite{CI}. Hence (3) holds by Proposition \ref{prop2}. \\
(ii)  Let $T = U \vert T \vert $ be $\log$-hyponormal and $f(t) = \log \, t$. Then $S = U \log |T|$ is semi-hyponormal and $\sigma (U \log |T|) = \{  e^{i\theta} \log \, r \, : \, r e^{i \theta} \in \sigma(T) \, \}$ by Lemma 8 of \cite{T}.  Hence (3) holds by Proposition \ref{prop2}.\\
Therefore,  if $T$ is $p$-hyponormal or $\log$-hyponormal and $f(t) = t^{2p}$ or $f(t) = \log \, t,$ respectively, then $T$ satisfies (2) and (3) for this $f$.

\bigskip
In this paper, we would like to prove the following theorem.

\begin{thm}\label{thm1} 

Let  ${\bf T} = (T_1,...,T_n)$ be a doubly commuting $n$-tuple of operators and $T_j= U_j|T_j| \ (j=1,...,n)$ be the polar decompositions. Let $f(t)$ be a continuous function on a open interval in the non-negative real line which contains $\sigma(|T_1|) \cup \cdots \cup \sigma(|T_n|)$. Let $S_j = U_j f( |T_j|) \ (j=1,...,n)$ and  ${\bf S} = (S_1,...,S_n)$. Let $T_1,..., T_n$ and $f$ satisfy $(1)$ and $(2)$. If  $(r_1 e^{i \theta_1},...,r_ne^{i \theta_n}) \in \sigma_T({\bf T}), \ \mbox{then }  \  (e^{i \theta_1} f(r_1),..., e^{i \theta_n} f(r_n))  \in \sigma_T({\bf S}) . $

\end{thm}

\section{Proof of the theorem } 

First we need the following lemma.

\noindent
\begin{lem}\label{lem1} Let ${\bf T} = (T_1,...,T_n)$ be a doubly commuting $n$-tuple of  operators and $T_j $ has property $(1)$ for $ j = 1, ... ,n$. Let $\{ D_k \}$ be the chain complex of $n$-tuple ${\bf T} = (T_1,...,T_n)$. If there exists some $ k \in \{ 1, 2, \cdots , n-1 \} $ and unit vectors $x_m = \oplus_{j=1}^r x_m^j \in E_k^n(\h)$ where $r =  \binom{n}{k}$, such that $(D_k^* D_k + D_{k+1} D_{k+1}^*)x_m \ \to \ 0$ as $m \to \infty$, then there exists $s \in \{ 1, 2, \cdots, r \}$  such that $\{x_m^s\}$ is a bounded below  sequence of non-zero vectors of $\h$ satisfying $T_j^* x_m^s \ \to \ 0$ as $m \to \infty$ for $j = 1, \cdots , n$.  Thus, by taking unit vector $y_{m} = \dfrac{ x_{m}^{s} }{ \Vert  x_{m}^{s} \Vert} \in {\mathcal H}$, we have  $T_j^* y_m \ \to \ 0$ as $m \to \infty$ for $j = 1, \cdots , n$. 
\end{lem}

\begin{proof}  We show it by the mathematical induction. \\
(1) Let $n=2$. Then the chain complex of doubly commuting pair ${\bf T}= (T_1,T_2)$ is 

$$ 0 \ \stackrel{ }{\ \longrightarrow \ } \ \h \stackrel{D_2}{\ \longrightarrow \ } \ \h \oplus \h \stackrel{D_1}{\ \longrightarrow \ } \ \h \stackrel{}{\ \longrightarrow \ } \ 0. $$

\

\noindent
By the definition of the Koszul complex we have 
$$D_2 = \left( 
\begin{array}{cc}
-T_2 \\
T_1
\end{array}
\right) \ \ \mbox{and} \ \ D_1 = \left( 
\begin{array}{cc}
T_1  & T_2 
\end{array}
\right).  $$

\noindent
Since $T_1, T_2$ are doubly commuting, we have

$$ \ \ D_1^* D_1 + D_2 D_2^* =  \left( 
\begin{array}{cccc}
T_1^* T_1 + T_2 T_2^* & 0 \\
0  & T_1 T_1^* + T_2^*T_2
\end{array}
\right). $$ 

\noindent
Let $x_m = x_m^1 \oplus x_m^2 \in E_1^2(\h) \cong \h \oplus \h$ be unit vectors and 
\begin{align*} & (D_1^* D_1 + D_2 D_2^*)x_m =  \begin{pmatrix} T_1^* T_1 + T_2 T_2^* & 0 \\
0  & T_1 T_1^* + T_2^*T_2 \end{pmatrix} \begin{pmatrix} x_m^1 \\ x_m^2 \end{pmatrix} \\
 & = \begin{pmatrix} (T_1^* T_1 + T_2 T_2^*)x_m^1 \\ (T_1 T_1^* + T_2^*T_2)x_m^2 \end{pmatrix}  \ \to \ 0 \ 
 \text{as} \ \  m \to \infty. \end{align*}
 
 Since $\|x_m^1\|^2 + \|x_m^2\|^2 = 1$ for all $m$, we may assume (i) $x_{m}^{1} \not\rightarrow 0 $ 
or (ii) $x_{m}^{2} \not\rightarrow 0 $. 

We assume (i).  By taking subsequence, we may asume that there exists $ 0 < c $ that that $ 0 < c <  \Vert x_{m}^{1} \Vert \leq 1$ for all $m$, i.e., bounded below.  Then $(T_1^* T_1 + T_2 T_2^*)x_m^1\rightarrow 0$ implies $T_{1} x_{m}^{1} , T_{2}^{*} x_{m}^{1} \rightarrow 0 $ and $T_{1}^{*} x_{m}^{1} \rightarrow 0 $ by (1). 
Case (ii) is similar.
Hence the statement holds for $n=2.$ \\

(2) We assume that the statement holds for $(n-1)$-tuples of doubly commuting operators. Asuume $(D_k^* D_k + D_{k+1} D_{k+1}^*)x_m \ \to \ 0$ as $m \to \infty$ for unit vectors $x_m \in E_k^n(\h)$.

Let $\{ F_k \}$ be the chain complex of  $(n-1)$-tuple ${\bf T}' = (T_1,...,T_{n-1})$ and  $x_m = y_m \oplus z_m \in E_{k}^{n-1}(\h) \oplus E_{k-1}^{n-1}(\h) = E_{k}^{n}(\h)$. By Curto's characterization (see p.132, Curto \cite{Cu}) it holds $D_k = \left( \begin{array}{cc}
F_k & (-1)^{k+1} {\rm diag} (T_n ) \\
0 & F_{k-1}
\end{array}
\right)$. Hence $$ (D_k^* D_k + D_{k+1} D_{k+1}^*)x_m  = \begin{pmatrix} \left( F_k^* F_k + F_{k+1} F_{k+1}^* + {\rm diag} (T_n T_n^*) \right) y_{m} \\ 
\left( F_{k-1}^* F_{k-1} + F_k F_k^* + {\rm diag} (T_n^* T_n) \right) z_{m} \end{pmatrix} \rightarrow 0.$$ 
 Since $\|y_m \|^2 + \|z_m \|^2 = 1$ for all $m$, we may assume (i) $y_{m} \not\rightarrow 0 $ 
or (ii) $z_{m} \not\rightarrow 0 $. 

We assume (i).

 Then $\left( F_k^* F_k + F_{k+1} F_{k+1}^* + {\rm diag} (T_n T_n^*) \right) y_{m} \rightarrow 0$ implies $\left( F_k^* F_k + F_{k+1} F_{k+1}^* \right) y_{m} \rightarrow 0$ and $\left(  {\rm diag} (T_n T_n^*) \right) y_{m} \rightarrow 0 $. 
By taking subsequence, we may asume that there exists $ 0 < c $ that that $ 0 < c <  \Vert y_{m} \Vert \leq 1$ for all $m$. Let $v_{m} = \dfrac{y_{m}}{ \Vert y_{m} \Vert} $. Then $v_{m}$ are 
unit vectors and  $\left( F_k^* F_k + F_{k+1} F_{k+1}^* \right) v_{m} \rightarrow 0$ and $\left(  {\rm diag} (T_n T_n^*) \right) v_{m} \rightarrow 0 $.  
Let $v_{m}  = \oplus_{s=1}^{ \binom{n-1}{k} } v_{m}^{s} \in E^{n-1}_{k} ( {\mathcal H} ) $. Then there exist $s \in \left\{1, 2,  \cdots , \binom{n-1}{k} \right\} $ such that $v_{m}^{s} \in {\mathcal H}$ is a bounded below sequence of non-zero vectors and $T_{j}^{*} v_{m}^{s} \rightarrow 0 $ for $ j = 1, 2, \cdots , n-1$ and $T_{n}^{*} v_{m}^{s} \rightarrow 0$ as $ m \rightarrow \infty$. 

 Case (ii) is similar.
Hence the statement holds for $n.$ It completes the proof. 
\end{proof}

\begin{thm}\label{thm2} Let ${\bf T} = (T_1,...,T_n)$ be a doubly commuting $n$-tuple of  operators which satisfy
 that every $T_j \ (j=1,...,n)$ has property $(1)$.
 If $z=(z_1,...,z_n) \in \sigma_T({\bf T})$, then there exists unit vectors $y_m \in {\mathcal H} $ such that $(T_j -z_j)^* y_m \ \to \ 0$ as $m \to \infty$, that is, $\overline{z} = (\overline{z_1},...,\overline{z_n}) \in \sigma_{ja}({\bf T}^*)$, where ${\bf T}^*=(T_1^*,...,T_n^*)$.  
\end{thm}

\begin{proof} Since $z=(z_1,...,z_n) \in \sigma_T({\bf T})$, by the spectral mapping theorem of the Taylor spectrum, it holds
$$0 = (0,...,0) \in \sigma_T({\bf T} -z),$$
where ${\bf T} -z=(T_1 - z_1,...,T_n - z_n)$. Since ${\bf T} -z$ is a doubly commuting $n$-tuple of operators which satisfy that every $T_j - z_j \ (j=1,...,n)$ has property (1) and the Koszul complex $E({\bf T} - z)$ of $n$-tuple ${\bf T}-z = (T_1 - z_1,...,T_n - z_n)$ is not exact. Hence there exists $k$ such that $(D_k^* D_k + D_{k+1} D_{k+1}^*)$ is not invertible. Since the operator $D_k^* D_k + D_{k+1} D_{k+1}^*$ is positive on the space $E_k^n(\h)$, there exists a sequence $\{x_m\}$ of unit vectors of $E_k^n(\h)$ such that $\big(D_k^* D_k + D_{k+1} D_{k+1}^*\big) x_m \ \to \ 0$ as $m \to \infty$. Hence, by Lemma \ref{lem1} there exists a sequence $\{y_m\}$ of unit vectors of $\h$ such that 
$$ (T_j - z_j)^*y_m \ \to \ 0 \ \ \mbox{as} \ m \  \to \ \infty \ \ \mbox{for all} \ \ j=1,...,n.$$
It's completes the proof. 
\end{proof}

\noindent
{\it Proof} of Theorem 1.9. \\
(1)  If $n=2$, theorem holds by Theorem 2.3 of \cite{CT}. \\
(2)  We assume that the statment holds for $(n-1)$-tuple. Since $(r_1 e^{i \theta_1},...,r_ne^{i \theta_n}) \in \sigma_T({\bf T})$, by Theorem \ref{thm2} there exists a sequence $\{x_m\}$ of unit vectors of $\h$ such that 
$(T_j - r_j e^{i \theta_j})^* x_m \ \to \ 0$ as $m \to \infty$ for all $j = 1,...,n$. Consider the Berberian extension ${\mathcal K}$  of $\h$. Then there exists $0 \not = x^{\circ} \in {\mathcal K}$  such that $$ (T_j^{\circ} - r_j e^{i\theta_j})^{*}  x^{\circ} = 0 \ \  \mbox{for all} \ \ j=1,...,n.$$ 
Let $ {\mathcal M} = \ker (T_n^{\circ}  - r_n e^{i\theta_n} )^*.$ Then $ {\mathcal M} (\not = \{0\})$ is a reducing subspace for $T_1^{\circ},...,T_{n-1}^{\circ}$ and $(r_1 e^{i \theta_1},...,r_{n-1}e^{i \theta_{n-1}}) \in \sigma_T({\bf T}_{|{\mathcal M}}^{\circ'})$, where ${\bf T}_{|{\mathcal M}}^{\circ'} = (T_{1|{\mathcal M}}^{\circ},...,T_{n-1|{\mathcal M}}^{\circ})$. By the induction there exists a non-zero vector $y^{\circ} \in {\mathcal M}$ such that  
$$ (S_j^{\circ} - e^{i \theta_j}f(r_j))^* y^{\circ} = 0 \ \ \mbox{for all} \ \ j=1,...,n-1.$$
Let $ {\mathcal N} = \displaystyle \bigcap_{j=1}^{n-1} \ker (S_j^{\circ}  - e^{i\theta_j} f(r_j ))^*.$ Then $ {\mathcal N}$ is a reducing subspace for $T_n^{\circ}$. Let  ${\mathcal R} = {\mathcal M} \bigcap {\mathcal N} \not = \{0\}$. Hence $r_n e^{i\theta_n} \in \sigma(T_{n|{\mathcal R}}^{\circ})$. By property (2) there exists a non-zero vector $z^{\circ} \in {\mathcal R}$ such that $(S_{n|{\mathcal R}}^{\circ} - e^{i\theta_n} f(r_n))^* z^{\circ} = 0$. Since this $z^{\circ}$ satisfies $(S_{j|{\mathcal R}}^{\circ} - e^{i\theta_j} f(r_j))^* z^{\circ} = 0$ for all $j=1,...,n-1$, we have $(e^{i\theta_1} f(r_1),...,e^{i\theta_n} f(r_n)) \in \sigma_T({\bf S})$. This completes the proof. 
\begin{flushright}
$\Box$
\end{flushright}

\begin{cor}\label{cor1} Let ${\bf T} = (T_1,..., T_n)$ be a doubly commuting $n$-tuple of $p$-hyponormal operators \ $(0<p< 1).$  Let $U_j$ be unitary for the polar decomposition of $T_j = U_j |T_j| \ (j=1,...,n)$ 
and ${\bf S} = \left( U_{1} \vert T_{1} \vert^{2p},...,  U_{n} \vert T_{n} \vert^{2p} \right)$. Then  
$$\sigma_T({\bf S}) \ = \ \{ (r_1^{2p} e^{i\theta_1},..., r_n^{2p} e^{i\theta_n}) \, : \, (r_1 e^{i\theta_1}, ...,r_n e^{i\theta_n}) \in \sigma_T({\bf T}) \, \}. $$
\end{cor}

\begin{proof} Let $f(t) = t^{2p}$ on the non-negative real line. Since ${\bf T}$ is a doubly commuting  $n$-tuple  of $p$-hyponormal operators and $f(t) = t^{2p}$, $T_1,...,T_n$ and $f$ satisfy  (2) and (3).  Hence, by Theorem \ref{thm1} we have 
$$\sigma_T({\bf S}) \ \supset \ \{ (r_1^{2p} e^{i\theta_1},..., r_n^{2p} e^{i\theta_n}) \, : \, (r_1 e^{i\theta_1},...,r_n e^{i\theta_n}) \in \sigma_T({\bf T}) \, \}. $$
Conversely, put  $g(t) = t^{\frac{1}{2p}}$ on the non-negative real line. Since ${\bf S}$ is a doubly commuting pair of semi-hyponormal operators, $S_1, S_2$ and $g$ satisfy (2) and (3). Then  we have the converse inclusion by Theorem \ref{thm1} and similar argument. 
\end{proof}

\begin{cor} Let ${\bf T} = (T_1,..., T_n)$ be a doubly commuting  $n$-tuple  of $\log$-hyponormal operators with $ \log \vert T_{j} \vert > 0$.  Let $U_j$ be unitary for the polar decomposition of $T_j = U_j |T_j| \ (j=1,...,n)$ 
and ${\bf S} = \left( U_{1} \log \vert T_{1} \vert, ..., U_{n}  \log \vert T_{n} \vert \right)$.   Then 
$$\sigma_T({\bf S}) = \{ e^{i\theta_1} \log r_1,...,  e^{i\theta_n} \log r_n) \, : \, (r_1 e^{i\theta_1},..., r_n e^{i\theta_n}) \in \sigma_T({\bf T}) \, \}. $$
\end{cor}

\begin{proof} Let $f(t) = \log t$ on $(0, \infty)$. Since ${\bf T}$ is a doubly commuting  $n$-tuple  of $\log$-hyponormal operators and $f(t) = \log t$, $T_1,...,T_n$ and $f$ satisfy $(2)$ and $(3)$. So by Theorem \ref{thm1} we have 
$$\sigma_T({\bf S})  \ \supset \ \{ e^{i\theta_1} \log r_1, ..., e^{i\theta_n} \log r_n) \, : \, 
(r_1 e^{i\theta_1},..., r_n e^{i\theta_n}) \in \sigma_T({\bf T}) \, \}. $$
Conversely, let  $g(t) = e^t $ on the non-negative real line. Since ${\bf S}$ is a doubly commuting  $n$-tuple of semi-hyponormal operators, $S_1,...,S_n$ and $g$ satisfy (2) and (3). Hence, we have the converse inclusion by similar argument. 
\end{proof}

\noindent
{\bf Acknowledgment.} This research is partially supported by RIMS.

\

\vspace{3mm}

\noindent
Muneo Ch\=o

\noindent
15-3-1113, Tsutsui-machi Yahatanishi-ku, Kita-kyushu 806-0032, Japan

\noindent
e-mail: muneocho0105@gmail.com
\vspace{3mm}

\noindent
B. Na\v cevska Nastovska

\noindent
Faculty of Electrical Engineering and Information Technologies, ''Ss Cyril and Methodius`` University in Skopje, Macedonia

\noindent
e-mail: bibanmath@gmail.com
\vspace{3mm}

\noindent
K\^otar\^o Tanahashi

\noindent
Department of Mathematics, Tohoku Medical and Pharmaceutical University, Sendai 981-8558, Japan

\noindent
e-mail: tanahasi@tohoku-mpu.ac.jp

\end{document}